\documentclass[12pt]{amsproc}

\usepackage{amsmath, amsthm, amssymb, amsfonts, mathrsfs, amscd}

\def\AA{{\mathbb A}}

\def\NN{{\mathbb N}}

\def\QQ{{\mathbb Q}}
\def\PP{{\mathbb P}}
\def\QQ{{\mathbb Q}}

\def\ZZ{{\mathbb Z}}

\def\0{{\mathbf 0}}
\def\1{{\mathbf 1}}

\def\Abf{{\mathbf A}}

\def\Fcal{{\mathcal F}}
\def\Gcal{{\mathcal G}}

\def\Lcal{{\mathcal L}}
\def\Mcal{{\mathcal M}}

\def\Ocal{{\mathcal O}}

\def\afrak{{\mathfrak a}}
\def\pfrak{{\mathfrak p}}

\def\Rfrak{{\mathfrak R}}

\def\Kbar{{\bar K}}

\def\Aut{\mathrm{Aut}}

\def\disc{\mathrm{disc}}

\def\Gal{\mathrm{Gal}}

\def\Crit{\mathrm{Crit}}

\def\ord{\mathrm{ord}}

\def\min{\mathrm{min}}

\theoremstyle{plain}

\newtheorem*{SzpConj}{Szpiro's Conjecture}
\newtheorem{thm}{Theorem}
\newtheorem{conj}{Conjecture}

\newtheorem{prop}[thm]{Proposition}
\newtheorem{lem}[thm]{Lemma}

\theoremstyle{definition}
\newtheorem*{dfn}{Definition}

\newtheorem*{rem}{Remark}
\newtheorem{ex}{Example}

\begin{document}

\title[Critically Separable Rational Maps in Families]{Critically Separable Rational Maps in Families}

\author{Clayton Petsche}

\address{Clayton Petsche; Department of Mathematics; Oregon State University; Corvallis OR 97331 U.S.A.}

\email{petschec@math.oregonstate.edu}

\thanks{Submitted September 13, 2011.  Revised January 17, 2012.  This research is supported by grant DMS-0901147 of the National Science Foundation.}



\begin{abstract}
Given a number field $K$, we consider families of critically separable rational maps of degree $d$ over $K$ possessing a certain fixed-point and multiplier structure.  With suitable notions of isomorphism and good reduction between rational maps in these families, we prove a finiteness theorem which is analogous to Shafarevich's theorem for elliptic curves.  We also define the minimal critical discriminant, a global object which can be viewed as a measure of arithmetic complexity of a rational map.  We formulate a conjectural bound on the minimal critical discriminant, which is analogous to Szpiro's conjecture for elliptic curves, and we prove that a special case of our conjecture implies Szpiro's conjecture in the semistable case.
\end{abstract}

\maketitle


\section{Introduction}\label{Background}

Let $K$ be a number field, let $M_K$ denote the set of places of $K$, and let $S$ be a finite subset of $M_K$ containing all of the Archimedean places.  A 1963 theorem of Shafarevich (\cite{silverman:aec} $\S$IX.6) states that there are only finitely many isomorphism classes of elliptic curves over $K$ having good reduction at all places $v\in M_K\setminus S$.  A generalization of this result to abelian varieties was proved by Faltings \cite{MR718935} in 1983, and, in combination with a result of Parshin, led to his proof of the Mordell conjecture.

Motivated by an analogy between elliptic curves and dynamical systems on the projective line, one might expect a similar finiteness result for rational maps $\phi:\PP^1_K\to\PP^1_K$.  The first to consider this problem were Szpiro-Tucker \cite{MR2435841}, who observed that, using the standard notions of isomorphism and good reduction for rational maps, simple counterexamples preclude a naive analogue of Shafarevich's theorem.  For example, rational maps defined by monic integral polynomials have everywhere good reduction, and for each fixed degree $d\geq2$ one can easily find infinite families of pairwise non-isomorphic maps of this type.  We will describe the work of Szpiro-Tucker in more detail below.

In order to describe our approach to this problem, we begin with an example of a family of rational maps which brings the elliptic curve analogy into sharper focus.  Fixing homogeneous coordinates $(x:y)$, we may identify $\PP^1_K$ with $\AA^1_K\cup\{\infty\}$, where $\infty=(1:0)$; this identifies each rational map $\phi:\PP^1_K\to\PP^1_K$ with a rational function $\phi(x)\in K(x)$ in the affine coordinate $x$.  Given a monic cubic polynomial $f(x)=x^3+ax^2+bx+c$, with coefficients in $K$ and with distinct roots in $\Kbar$, define a rational map $\phi_{a,b,c}:\PP_K^1\to\PP_K^1$ by
\begin{equation}\label{LattesExample}
\phi_{a,b,c}(x) = \frac{x^4-2bx^2-8cx+b^2-4ac}{4x^3+4ax^2+4bx+4c}.
\end{equation}
The significance of this rational map lies in its correspondence with the elliptic curve $E/K$ defined by the Weierstrass equation $y^2=x^3+ax^2+bx+c$.  Let $[2]:E\to E$ denote the doubling map $P\mapsto 2P=P+P$, and let $x:E\to\PP^1_K$ denote the $x$-coordinate map.  Then the rational map $\phi_{a,b,c}$, which is called a Latt\`es map, completes (\cite{silverman:aec} $\S$III.2) the commutative diagram
\begin{equation}\label{LattesMap}
\begin{CD}
E  @> [2] >>   E \\ 
@V x VV                                    @VV x V \\ 
\PP_K^1           @> \phi_{a,b,c} >>      \PP_K^1
\end{CD} 
\end{equation}
Denote by $\Lcal(K)$ the family of all such rational maps $\phi_{a,b,c}$ defined over $K$.  Consider the following list of properties of the family $\Lcal(K)$:

\begin{itemize}
	\item[(L1)] Each rational map $\phi_{a,b,c}\in\Lcal(K)$ has degree $4$.  
	\item[(L2)] The point $\infty$ is an unramified fixed point of each rational map $\phi_{a,b,c}\in\Lcal(K)$, with multiplier $4$.  
	\item[(L3)] The numerator of each rational map $\phi_{a,b,c}\in\Lcal(K)$ has vanishing $x^3$ term.  
	\item[(L4)] Each rational map $\phi_{a,b,c}\in\Lcal(K)$ has six distinct critical points in $\PP^1(\Kbar)$, which is the highest number allowed for a rational map of degree $4$ by the Riemann-Hurwitz formula.
\end{itemize}
We will discuss the family $\Lcal(K)$ in more detail in $\S$\ref{GenDynSysSect}.

In this paper, our primary objects of study are certain families of rational maps whose definitions generalize properties (L1)-(L4) of the family $\Lcal(K)$ of Latt\`es maps.  Our main result is a finiteness theorem for isomorphism classes of rational maps, varying in such families, which satisfy a certain strong form of good reduction at all places $v\in M_K\setminus S$.  A special case of our main result implies such a finiteness statement for the family $\Lcal(K)$ of Latt\`es maps; this result is essentially equivalent to Shafarevich's theorem, in the sense that each statement can be easily deduced from the other.

To state our results, we require some notation and some definitions.  Given an integer $d\geq2$ and a nonzero element $\lambda\in K^\times$, consider a rational map $\phi:\PP^1_K\to\PP^1_K$ of degree $d$ such that $\infty$ is a fixed point of $\phi$ with multiplier $\lambda$.  In the affine coordinate $x$, such a rational map can be written uniquely as
\begin{equation}\label{RatMapAffine}
\phi(x) = \frac{x^d+a_{d-1}x^{d-1}+\dots+a_0}{\lambda x^{d-1}+b_{d-2}x^{d-2}+\dots+b_0}
\end{equation}
for coefficients $a_j,b_j\in K$, where the numerator and denominator have no common roots in $\Kbar$.  According to the Riemann-Hurwitz formula, when counted with multiplicity, the rational map $\phi$ has exactly $2d-2$ critical points in $\PP^1(\Kbar)$.  We say that $\phi$ is {\em critically separable} if it has $2d-2$ distinct critical points in $\PP^1(\Kbar)$.  We will see in $\S$\ref{GenDynSysSect} that a generic rational map of the form $(\ref{RatMapAffine})$ has degree $d$ and is critically separable.

\begin{dfn}
Given an integer $d\geq2$ and an element $\lambda\in K^\times$, define $\Fcal_{d,\lambda}(K)$ to be the family of all rational maps $\phi:\PP^1_K\to\PP^1_K$ satisfying
\begin{itemize}
	\item[(F1)] $\deg(\phi)=d$;
	\item[(F2)] $\infty$ is a fixed point of $\phi$ with multiplier $\lambda$;
	\item[(F3)] $a_{d-1}=\epsilon b_{d-2}$, where $\epsilon=(d-\lambda)/(d-1)\lambda$;
	\item[(F4)] $\phi$ is critically separable.
\end{itemize} 
\end{dfn}

The definition of the space $\Fcal_{d,\lambda}(K)$ is partially inspired by the aforementioned properties of the space $\Lcal(K)$ of Latt\`es maps.  In fact, comparison of the four properties (L1)-(L4) of the family $\Lcal(K)$ with the corresponding parts (F1)-(F4) in the definition of $\Fcal_{d,\lambda}(K)$ shows that $\Lcal(K)$ is a (proper) subfamily of $\Fcal_{4,4}(K)$.

Very little is lost in considering only those rational maps fixing $\infty$, for if $\phi:\PP_K^1\to\PP_K^1$ is an arbitrary rational map, then possibly after replacing $K$ with a finite extension of $K$, there exists a point $P\in\PP^1(K)$ such that $\phi(P)=P$.  Replacing $\phi$ with $\sigma\circ\phi\circ\sigma^{-1}$ for a suitably chosen $\sigma\in\Aut(\PP^1_K)$, we may assume without loss of generality that $P=\infty$.

\begin{dfn}  Let $\Aut^\infty(\PP^1_K) = \{ x\mapsto \alpha x+\beta \mid \alpha\in K^\times, \beta\in K\}$.  We say that two rational maps $\phi,\psi\in\Fcal_{d,\lambda}(K)$ are {\em isomorphic} (over $K$) if there exists $\sigma\in\Aut^\infty(\PP^1_K)$ such that $\sigma\circ\phi\circ\sigma^{-1}=\psi$.  
\end{dfn}

Note that $\Aut^\infty(\PP^1_K)$ is precisely the subgroup of $\Aut(\PP^1_K)$ consisting of those automorphisms which fix $\infty$, and so in view of condition (F2), conjugation by the group $\Aut^\infty(\PP^1_K)$ is a natural notion of isomorphism between rational maps in $\Fcal_{d,\lambda}(K)$.  It is not hard to see that each of the conditions (F1)-(F4) is invariant under $\Aut^\infty(\PP^1_K)$-conjugation, and thus the family $\Fcal_{d,\lambda}(K)$ is closed under isomorphism.

It is instructive at this point to revisit the analogy with elliptic curves.  Recall that an elliptic curve over $K$ is defined to be a pair $(X,O)$, where $X$ is a complete nonsingular curve of genus one over $K$, and where $O$ is a $K$-rational point on $X$ which acts as origin for the group law on $X(K)$.  An isomorphism between two elliptic curves $(X_1,O_1)$ and $(X_2,O_2)$ is an isomorphism $X_1\to X_2$ of curves with $O_1\mapsto O_2$.   Thus, the difference between an $\Aut(\PP^1_K)$-conjugation class of rational maps and an isomorphism class of rational maps in the family $\Fcal_{d,\lambda}(K)$ is analogous to the difference between an isomorphism class of curves of genus one over $K$ and an isomorphism class of elliptic curves over $K$.  It is also worth mentioning, in view of our main result, Theorem~\ref{FinitenessTheoremIntro}, that Shafarevich's Theorem would be false in general if ``elliptic curve'' were replaced by ``curve of genus one''; see Mazur \cite{MR828821} p. 241.

Condition (F3) in the definition of the family $\Fcal_{d,\lambda}(K)$ is a natural generalization of the observation (L3) concerning the family $\Lcal(K)$ of Latt\`es maps.  Given a rational map $\phi\in\Fcal_{d,\lambda}(K)$, written as in $(\ref{RatMapAffine})$, let us call $\phi$ {\em centered} if both $a_{d-1}=0$ and $b_{d-2}=0$.  It is not hard to see that every isomorphism class in $\Fcal_{d,\lambda}(K)$ contains a rational map $\phi$ with $b_{d-2}=0$ (this observation is analogous to the fact that every elliptic curve $E/K$ has a Weierstrass equation of the form $y^2=x^3+bx+c$), and condition (F3) ensures that such a rational map in $\Fcal_{d,\lambda}(K)$ satisfies $a_{d-1}=0$ as well; that is, such a rational map is centered.  The choice of $\epsilon=(d-\lambda)/(d-1)\lambda$ ensures that the condition (F3) is invariant under $\Aut^\infty(\PP^1_K)$-conjugation; this follows from a simple calculation of the effect of $\Aut^\infty(\PP^1_K)$-conjugation on the coefficients $a_{d-1}$ and $b_{d-2}$.  Thus $\Fcal_{d,\lambda}(K)$ could be described as the smallest family of critically separable rational maps for which $\infty$ is a fixed point of multiplier $\lambda$, which contains all of the centered rational maps, and which is closed under $\Aut^\infty(\PP^1_K)$-conjugation.

To further emphasize the necessity of conditions (F2) and (F3) in the definition of the family $\Fcal_{d,\lambda}(K)$, we remark that the primary theme of our main result, Theorem~\ref{FinitenessTheoremIntro}, is the recovery of information about a rational map from knowledge of its critical locus.  Any such result must respect the fact that if $\phi:\PP^1_K\to\PP^1_K$ is a rational map and $\sigma\in \Aut(\PP^1_K)$ is an automorphism, then $\phi$ and $\sigma\circ\phi$ share the same critical locus.  Together, conditions (F2) and (F3) ensure that $\phi$ and $\sigma\circ\phi$ cannot both belong to $\Fcal_{d,\lambda}(K)$ unless $\sigma$ is the trivial automorphism; this fact forms the technical heart of Lemma~\ref{RigidityLemma}.  Simple counterexamples show that Theorem~\ref{FinitenessTheoremIntro} would be false if one of the  conditions (F2) or (F3) were omitted.  

On the other hand, it is possible to modify conditions (F2) and (F3) to produce other potentially interesting families of critically separable rational maps for which a version of our main finiteness result can be proved, using essentially the same argument.  To illustrate this point, we will give an example of such a family at the end of $\S$\ref{FinitenessThmSect}.

Before we can state our main result we must define what we mean by ``good reduction'' of a rational map in the family $\Fcal_{d,\lambda}(K)$.  For each non-Archimedean place $v$ of $K$, let $\Ocal_v$ denote the ring of $v$-integral elements of $K$, let $\Mcal_v$ denote the maximal ideal of $\Ocal_v$, and let $k_v=\Ocal_v/\Mcal_v$ denote the residue field.  We say $\phi\in\Fcal_{d,\lambda}(K)$ is {\em $v$-integral} if, when written as in $(\ref{RatMapAffine})$, the coefficients $a_j, b_j,\lambda$ are elements of $\Ocal_v$.  In this case, reducing the coefficients modulo $\Mcal_v$ we may meaningfully define a reduced rational map $\tilde{\phi}_v:\PP^1_{k_v}\to\PP^1_{k_v}$.

\begin{dfn}  Let $v$ be a non-Archimedean place of $K$.  A rational map $\phi\in\Fcal_{d,\lambda}(K)$ has {\em critically separable good reduction} at $v$ if it is $K$-isomorphic to a $v$-integral rational map $\psi\in\Fcal_{d,\lambda}(K)$ such that the reduced rational map $\tilde{\psi}_v:\PP^1_{ k_v}\to\PP^1_{ k_v}$ has degree $d$ and is critically separable.
\end{dfn}

Note that all rational maps in $\Fcal_{d,\lambda}(K)$ automatically have critically separable bad reduction at all places $v$ for which $\lambda\not\in\Ocal_v$.  We are now ready to state our main result.

\begin{thm}\label{FinitenessTheoremIntro}
Let $S$ be a finite set of places of the number field $K$ including all of the Archimedean places, let $d\geq2$ be an integer, and let $\lambda\in K^\times$.  Then the family $\Fcal_{d,\lambda}(K)$ contains only finitely many $K$-isomorphism classes of rational maps having critically separable good reduction at all places $v\not\in S$.
\end{thm}

The proof of Theorem~\ref{FinitenessTheoremIntro} relies ultimately on Diophantine approximation, namely the standard result on the finiteness of $S$-integral solutions to the unit equation $x+y=1$ (see \cite{bombierigubler} $\S$5.1).  This should not be surprising to those familiar with any of the usual proofs of Shafarevich's theorem (see for example \cite{silverman:aec} $\S$IX.6), which rely on the closely related finiteness result of Siegel for integral points on curves of genus at least one.  The second major ingredient in our proof of Theorem~\ref{FinitenessTheoremIntro} is a classical finiteness theorem (see \cite{MR1093002}) for rational maps with a prescribed critical locus; we will describe this result in more detail in the proof of Lemma~\ref{RigidityLemma}.

We will see in $\S$\ref{GenDynSysSect} that a rational map $\phi:\PP^1_K\to\PP^1_K$ written as in $(\ref{RatMapAffine})$ has degree $d$ and is critically separable if and only if its {\em critical discriminant}, a certain polynomial expression in the coefficients $a_j$ and $b_j$, is nonvanishing.  Consequently, the notion of critically separable good reduction can be detected by the critical discriminant of a rational map, in much the same way that the discriminant of a Weierstrass equation detects good reduction of an elliptic curve.  Taking the analogy a step further, in $\S$\ref{MinCritDisc} we will define the {\em minimal critical discriminant} of $\phi$, an integral ideal of $\Ocal_K$ which is supported on the places at which $\phi$ has critically separable bad reduction, and which can be viewed as one measure of the arithmetic complexity of $\phi$.  By analogy with Szpiro's conjecture for the minimal discriminant of an elliptic curve, in $\S$\ref{MinCritDisc} we will propose a conjectural bound on the size of the minimal critical discriminant of $\phi$ in terms of the set of places at which $\phi$ has critically separable bad reduction.  We will show in Theorem~\ref{ImpliesSzpiro} that our conjecture for the family $\Fcal_{4,4}(K)$ implies Szpiro's conjecture for semistable elliptic curves. 

This research was inspired in part by the paper \cite{MR2435841} of Szpiro-Tucker, who were the first to prove an analogue of Shafarevich's theorem for rational maps.  Our Theorem~\ref{FinitenessTheoremIntro} is similar in spirit to their main result, and we borrow several key ideas from their paper, notably the use of the critical locus to define a notion of good reduction, and the use of the $S$-unit equation via results such as \cite{MR0306119} and our Theorem~\ref{BirchMerrimanTheoremAffine}.  However, the formulations of our Theorem~\ref{FinitenessTheoremIntro} and the main result of \cite{MR2435841} are sufficiently different that neither theorem is stronger than the other.  It is a strength of \cite{MR2435841} that its main finiteness result holds over all rational maps of degree $d$ possessing at least three critical points, while our Theorem~\ref{FinitenessTheoremIntro} only gives a finiteness result along each family $\Fcal_{d,\lambda}(K)$ of critically separable rational maps.  On the other hand, within this more modest framework our result has the following two advantages.  First, in \cite{MR2435841}, isomorphism between rational maps is defined by the equivalence $\phi\sim\psi$ whenever $\phi=\sigma\circ\psi\circ\tau$ for $\sigma,\tau\in\Aut(\PP_K^1)$; in other words, their definition uses independent pre-composition and post-composition actions of the automorphism group of $\PP^1_K$.  In contrast, our notion of isomorphism for the family $\Fcal_{d,\lambda}(K)$, defined by the conjugation equivalence $\phi\sim\psi$ whenever $\phi=\sigma\circ\psi\circ\sigma^{-1}$ for automorphisms $\sigma\in\Aut^\infty(\PP_K^1)$, is a more natural choice in the context of dynamics because it is better behaved under iteration.  Second, in \cite{MR2435841}, the notion of {\em critically good reduction} of a rational map is neither stronger nor weaker than standard good reduction, and it relies on behavior of both the critical locus and the branch locus.  Our notion of {\em critically separable good reduction} is strictly stronger than standard good reduction, and it relies only on behavior of the critical locus.  Moreover, our notion of critically separable good reduction is detected by the critical discriminant, which leads to the minimal critical discriminant and in turn to Conjecture~\ref{CondDiscConjecture}, an analogue of Szpiro's conjecture for critically separable rational maps.

Silverman \cite{MR2316407} and Szpiro-Tepper-Williams \cite{arxiv1010.5030} have considered the {\em minimal resultant} associated to a rational map $\phi:\PP_K^1\to\PP_K^1$.  This is an integral ideal of $\Ocal_K$ which is supported on the places at which $\phi$ has bad reduction in the standard sense, and like our minimal critical discriminant, it can be viewed as an analogue for rational maps of the minimal discriminant of an elliptic curve.  Szpiro-Tepper-Williams \cite{arxiv1010.5030} have given counterexamples to show that the minimal resultant is not bounded solely in terms of the set of places at which $\phi$ has bad reduction; on the other hand, they have proposed a conjecture stating that it can be bounded in terms of the set of places at which $\phi$ has {\em critically bad reduction} in the sense of Szpiro-Tucker \cite{MR2435841}.

The plan of this paper is the following:  In $\S$\ref{GenDynSysSect} we will define the key technical tool of the paper, the critical discriminant, and discuss its properties.  In $\S$\ref{FinitenessThmSect} we will prove a number of preliminary number-theoretic results, and we will give the proof of Theorem~\ref{FinitenessTheoremIntro}.  Finally, in $\S$\ref{MinCritDisc} we will define the minimal critical discriminant of a rational map in the family $\Fcal_{d,\lambda}(K)$, state Conjecture~\ref{CondDiscConjecture}, and discuss its relationship to Szpiro's conjecture.

We would like to acknowledge Aaron Levin for bringing the aforementioned passage in \cite{MR828821} to our attention, and the anonymous referee for his or her many excellent suggestions.

\section{The Critical Discriminant}\label{GenDynSysSect}

For this section only, $K$ denotes an arbitrary field (not necessarily a number field).  We begin by reviewing a few basic facts about discriminants of polynomials; for details see \cite{bombierigubler} $\S$B.1.  Given a polynomial $P(x) \in K[x]$ of degree $N$, the discriminant $\disc(P)$ is an integer polynomial in the coefficients of $P(x)$ which can be defined as the determinant of a certain Sylvester matrix.  Alternatively, factoring $P(x)=a\prod_n(x-r_n)$ for $a\in K^\times$, $r_n\in \Kbar$, the discriminant is given by
\begin{equation}\label{DiscDef}
\disc(P)=a^{2N-2}\prod_{m<n}(r_m-r_n)^2.
\end{equation}
It is evident from $(\ref{DiscDef})$ that $\disc(P)\neq0$ if and only if $P(x)$ has $N$ distinct roots, and that 
\begin{equation}\label{DiscScale}
\disc(\lambda P)=\lambda^{2N-2}\disc(P)
\end{equation}
for all $\lambda\in K^\times$.  Given an automorphism $\sigma\in\Aut^\infty(\PP^1_K)$, written as $\sigma(x)=\alpha x+\beta$ for $\alpha\in K^\times$ and $\beta\in K$, it follows from an elementary calculation using $(\ref{DiscDef})$ that 
\begin{equation}\label{FormTrans}
\disc(P_\sigma)=\alpha^{N(N-1)}\disc(P),
\end{equation}
where $P_\sigma(x)=P(\sigma(x))=P(\alpha x+\beta)$.

Let $d\geq2$ be an integer, and let $\lambda\in K^\times$.  We say an ordered pair $(A(x),B(x))$ of polynomials in $K[x]$ is in {\em standard form} with respect to the pair $(d,\lambda)$ if
\begin{equation*}
\begin{split}
A(x) & = x^d+a_{d-1}x^{d-1}+\dots+a_0 \\
B(x) & = \lambda x^{d-1}+b_{d-2}x^{d-2}+\dots+b_0
\end{split}
\end{equation*}
for coefficients $a_j,b_j\in K$; in other words, $A(x)$ must have degree $d$ and be monic, and $B(x)$ must have degree $d-1$ and leading coefficient $\lambda$.  Given such a pair, the rational map $\phi:\PP^1_K\to\PP^1_K$ defined by $\phi(x)=A(x)/B(x)$ has degree at most $d$, with $\deg(\phi)=d$ if and only if $A(x)$ and $B(x)$ have no common roots in $\Kbar$.  Moreover, $\infty$ is a fixed point of $\phi$ with multiplier $\lambda$.  

Conversely, an arbitrary rational map $\phi:\PP^1_K\to\PP^1_K$ of degree $d$ for which $\infty$ is a fixed point with multiplier $\lambda$ can be written (uniquely) in the affine coordinate $x$ as $\phi(x)=A(x)/B(x)$ for a pair $(A(x),B(x))$ of polynomials in standard form.

Define the {\em Wronskian} of the pair $(A(x),B(x))$ to be the polynomial
\begin{equation}\label{FormE}
W_{A,B}(x) = B(x)A'(x) - A(x)B'(x);
\end{equation}
thus the derivative of $A(x)/B(x)$ is $W_{A,B}(x)/B(x)^2$.  Observe that $W_{A,B}(x)=\lambda x^{2d-2}+\dots$, and thus $\deg(W_{A,B})=2d-2$.  Define the {\em critical discriminant} of the pair $(A(x),B(x))$ by
\begin{equation}\label{CritDisc}
\Delta_{A,B}=\disc(W_{A,B}).
\end{equation}

The significance and basic properties of the Wronskian $W_{A,B}(x)$ and the critical discriminant $\Delta_{A,B}$ are explained in the following proposition.  The most important property is part (c), which states that the critical discriminant $\Delta_{A,B}$ is nonvanishing if and only if the corresponding rational map $\phi(x)=A(x)/B(x)$ has degree $d$ and is critically separable.

\begin{prop}\label{GenDiscProp}
Let $d\geq2$ be an integer, let $\lambda\in K^\times$, and let $(A(x),B(x))$ be a pair of polynomials in standard form with coefficients in $K$.  Denote by $\phi:\PP^1_K\to\PP^1_K$ the rational map defined by $\phi(x)=A(x)/B(x)$.
\begin{enumerate}
\item[(a)]  If $r\in\Kbar$ is a common root of $A(x)$ and $B(x)$, then $r$ is at least a double root of $W_{A,B}(x)$.

\item[(b)]  If $\deg(\phi)=d$ and $r\in\Kbar$, then $W_{A,B}(r)=0$ if and only if $r$ is a critical point of $\phi$. 

\item[(c)]  $\Delta_{A,B}\neq0$ if and only if $\deg(\phi)=d$ and $\phi$ has $2d-2$ distinct critical points in $\Kbar$.

\item[(d)]  Given $\sigma\in\Aut^\infty(\PP^1_K)$, written as $\sigma(x)=\alpha x+\beta$ for $\alpha\in K^\times$ and $\beta\in K$, the rational map $\sigma\circ\phi\circ\sigma^{-1}:\PP_K^1\to\PP_K^1$ is given by $\sigma\circ\phi\circ\sigma^{-1}(x)=A^\sigma(x)/B^\sigma(x)$ for polynomials
\begin{equation}\label{ABSigma}
\begin{split}
A^\sigma(x) & = \alpha^d A(\alpha^{-1}(x-\beta))+\alpha^{d-1}\beta B(\alpha^{-1}(x-\beta)) \\
B^\sigma(x) & = \alpha^{d-1}B(\alpha^{-1}(x-\beta))
\end{split}
\end{equation}
in standard form, and
\begin{equation}\label{DeltaChangeSigma}
\Delta_{A^\sigma,B^\sigma} = \alpha^{(2d-2)(2d-3)}\Delta_{A,B}.
\end{equation}
\end{enumerate}
\end{prop}

\begin{proof}
(a)  If $A(x)=(x-r)A_0(x)$ and $B(x)=(x-r)B_0(x)$, then an elementary calculation shows that 
\begin{equation*}
W_{A,B}(x) = (x-r)^2(B_0(x)A_0'(x)-A_0(x)B_0'(x)).
\end{equation*}

(b)  Assume that $\deg(\phi)=d$ (thus $A(x)$ and $B(x)$ have no common roots in $\Kbar$), and let $r\in\Kbar$.  Case 1: $B(r)\neq0$.  In this case a standard calculation shows that 
\begin{equation*}
\phi(x)-\phi(r)=\frac{W_{A,B}(r)}{B(r)^2}(x-r) + (x-r)^2\psi(x)
\end{equation*}
for a rational map $\psi:\PP_K^1\to\PP_K^1$ with $\psi(r)\neq\infty$; it follows that $r$ is a critical point of $\phi$ if and only if $W_{A,B}(r)=0$, completing the proof in case 1.  

Case 2: $B(r)=0$.  In this case $A(r)\neq0$, and we consider the rational map $\phi_0:\PP_K^1\to\PP_K^1$ defined by $\phi_0(x)=B(x)/A(x)$.  Since $\phi_0=\sigma\circ\phi$ for the involution $\sigma\in\Aut(\PP_K^1)$ defined by $\sigma(x)=1/x$, it follows that $r$ is a critical point of $\phi$ if and only if it is a critical point of $\phi_0$.  By case 1, $r$ is a critical point of $\phi_0$ if and only if $W_{B,A}(r)=0$, and since $W_{A,B}(x)=-W_{B,A}(x)$, we have $W_{B,A}(r)=0$ if and only if $W_{A,B}(r)=0$, completing the proof of case 2.

(c)  If $\Delta_{A,B}\neq0$, then $W_{A,B}(x)$ has $2d-2$ distinct roots in $\Kbar$, which implies that $A(x)$ and $B(x)$ have no common roots in $\Kbar$ by part (a), whereby $\deg(\phi)=d$.  Part (b) implies that $\phi$ has $2d-2$ distinct critical points.  Conversely, if $\deg(\phi)=d$ and $\phi$ has $2d-2$ distinct critical points, then part (b) implies that $W_{A,B}(x)$ has $2d-2$ distinct roots in $\Kbar$, whereby $\Delta_{A,B}\neq0$. 

(d)  The calculation of the polynomials $A^\sigma(x)$ and $B^\sigma(x)$ is elementary.  It is easy to see that $W_{A^\sigma,B^\sigma}(x) = \alpha^{2d-2}W_{A,B}(\alpha^{-1}(x-\beta))$, and combining this fact with the properties $(\ref{DiscScale})$ and $(\ref{FormTrans})$ of discriminants, one arrives at the identity $(\ref{DeltaChangeSigma})$.  
\end{proof} 

\begin{ex}\label{QuadEx}
Let $\phi(x)=A(x)/B(x)$ for $A(x)=x^2+ax+b$ and $B(x)=\lambda x+c$.   Then $W_{A,B}(x)=\lambda x^2 + 2cx+(ac-\lambda b)$, which has discriminant
\begin{equation*}
\Delta_{A,B}=4c^2-4\lambda(ac-\lambda b).
\end{equation*}
\end{ex}

\begin{ex}
Returning to the family $\Lcal(K)$ of Latt\`es maps described in $\S$\ref{Background}, let $f(x)=x^3+ax^2+bx+c$ be a monic cubic polynomial, with coefficients in $K$ and with distinct roots in $\Kbar$, and let $\phi_{a,b,c}:\PP_K^1\to\PP_K^1$ be the Latt\`es map $(\ref{LattesExample})$ associated to the elliptic curve $E$ defined by $y^2=f(x)$.  Thus $\phi_{a,b,c}(x)=A(x)/B(x)$, where $A(x)=x^4-2bx^2-8cx+b^2-4ac$ and $B(x)=4x^3+4ax^2+4bx+4c$.  

We now elaborate briefly on properties (L1)-(L4) of the family $\Lcal(K)$, as listed in $\S$\ref{Background}.  Property (L1) follows at once from the diagram $(\ref{LattesMap})$ and the fact that $\deg(x)=2$ and $\deg([2])=4$.  Properties (L2) and (L3) are self-evident.  To see property (L4), observe that the map $x:E\to\PP^1_K$ is an even double cover, ramified only at the four $2$-torsion points of $E$, and the map $[2]:E\to E$ is unramified.  These facts and inspection of the diagram $(\ref{LattesMap})$ show that the critical locus of $\phi_{a,b,c}$ is precisely $x(E[4]\setminus E[2])$, where $E[n]$ denotes the set of $n$-torsion points in $E(\Kbar)$.  The set $E[4]\setminus E[2]$ consists of twelve points occuring in six pairs $\pm P_1, \dots,\pm P_6$, and the critical locus of $\phi$ consists of the six distinct points $x(P_1),\dots,x(P_6)$.

Not surprisingly, the critical discriminant $\Delta_{A,B}$ is closely related to the discriminant $\Delta_E$ of the Weierstrass equation $y^2=f(x)$.  Recall (\cite{silverman:aec} $\S$III.1) that the latter is given by 
\begin{equation}\label{EllipticDisc}
\Delta_E=2^4\disc(f),
\end{equation}
where
\begin{equation*}
\disc(f) = a^2b^2+18abc-4a^3c-4b^3-27c^2 
\end{equation*}
is the discriminant of the cubic polynomial $f(x)$.  We will see that
\begin{equation}\label{CritDiscLattes}
\Delta_{A,B} = -2^{38}\disc(f)^5. 
\end{equation}

One could simply blast out both sides of $(\ref{CritDiscLattes})$ and check that they are equal.  But the following more conceptual argument is perhaps more illuminating, and it reduces the calculation to a simpler special case.  Viewing $a, b, c$ as variables, $\Delta_{A,B}$ and $\disc(f)$ are elements of the polynomial ring $\ZZ[a,b,c]$ which vanish on precisely the same set of $(a,b,c)$ in $\Kbar^3$.  For if $\disc(f)\neq0$, then the discussion in $\S$\ref{Background} shows that the map $\phi_{a,b,c}$ has degree $4$ and is critically separable, and so $\Delta_{A,B}\neq0$ follows via Proposition~\ref{GenDiscProp} (c).  Conversely, if $\disc(f)=0$, then $f(x)$ has a double root in $\Kbar$, say $r$.  Then plainly $B(r)=4f(r)=0$, and the easily checked identity $A(x)=f'(x)^2-(8x+4a)f(x)$ shows that $A(r)=0$ as well.  This means that $\deg(\phi_{a,b,c})<4$, and consequently $\Delta_{A,B}=0$ using Proposition~\ref{GenDiscProp} (c).  

Since the elements $\Delta_{A,B}$ and $\disc(f)$ of $\ZZ[a,b,c]$ vanish simultaneously, and since the latter is irreducible, it follows that $\Delta_{A,B} = q\cdot\disc(f)^{n}$ for some $q\in \QQ^\times$ and some integer $n\geq1$.  Given $\alpha\in K^\times$, consider the monic polynomial $f^*(x)=\alpha^3f(\alpha^{-1}x)$, and let $A^*(x)$ and $B^*(x)$ be the numerator and denominator of the Latt\`es map corresponding as above to the elliptic curve $y^2=f^*(x)$.  Calculations show that $\disc(f^*)=\alpha^6\disc(f)$ and $\Delta_{A^*,B^*}=\alpha^{30}\Delta_{A,B}$, and since $\Delta_{A^*,B^*} = q\cdot\disc(f^*)^{n}$, we must have $n=5$.  To calculate $q$, consider the case $a=0$, $b=1$, $c=0$; thus $f(x)=x^3+x$ and $\disc(f)=-4$.  In this case $W_{A,B}(x)=4x^6-20x^4-20x^2-4$, which has discriminant $\Delta_{A,B}=2^{48}$.  It follows that $q=-2^{38}$.
\end{ex}

\section{The Finiteness Theorem}\label{FinitenessThmSect}

For the remainder of this paper $K$ denotes a number field.  Let $M_K$, $M_K^\infty$, and $M_K^0$ denote the set of all places, all Archimedean places, and all non-Archimedean places of the number field $K$, respectively.  Given a subring $R$ of an extension field of $K$, define 
\begin{equation*}
\Aut^\infty(\PP^1_R) = \{ x\mapsto \alpha x+\beta \mid \alpha \in R^\times, \beta\in R\}.
\end{equation*}

\begin{lem}\label{ClassNumberLemma}
Given a number field $K$, there exists a finite subset $S_0$ of $M_K$ containing $M_K^\infty$ with the following property.  If $S$ is a finite subset of $M_K$ containing $S_0$, and if $\sigma_v\in \Aut^\infty(\PP^1_K)$ for each $v\in M_K\setminus S$, such that $\sigma_v\in \Aut^\infty(\PP^1_{\Ocal_v})$ for all except finitely many places $v$, then there exists some $\sigma\in \Aut^\infty(\PP^1_K)$ such that $\sigma\sigma_v^{-1}\in \Aut^\infty(\PP^1_{\Ocal_v})$ for all $v\in M_K\setminus S$.
\end{lem}
\begin{proof}
For each place $v\in M_K$, denote by $\widehat{K}_v$ the completion of $K$ at $v$, and if $v$ is non-Archimedean let $\widehat{\Ocal}_v$ denote the ring of $v$-integral elements of $\widehat{K}_v$.  

Let $G(K)$ denote the affine algebraic group $\Aut^\infty(\PP^1_K)$, and let $G(\Abf_K)$ be the adele group associated to $G(K)$.  Thus $G(\Abf_K)$ is the subgroup of the direct product of the groups $G(\widehat{K}_v)$, indexed over all places $v\in M_K$, where an element $(\sigma_v)$ of this product is in $G(\Abf_K)$ if and only if $\sigma_v\in G(\widehat{\Ocal}_v)$ for all except finitely many $v\in M_K$.  Recall that $G(K)$ is naturally identified with the subgroup of principal adeles in $G(\Abf_K)$.  Denote by $G^\infty(\Abf_K)$ the subgroup of $G(\Abf_K)$ consisting of those $(\sigma_v)\in G(\Abf_K)$ with $\sigma_v\in G(\widehat{\Ocal}_v)$ for all $v\in M_K^0$.

A theorem of Borel (\cite{MR0202718}, Thm 5.1) states that $G(\Abf_K)$ is equal to a finite union 
\begin{equation}\label{AdeleFinite}
G(\Abf_K)=\bigcup_{1\leq n\leq N}(G^\infty(\Abf_K)\cdot \sigma_n\cdot G(K))
\end{equation}
of double cosets by the two subgroups $G^\infty(\Abf_K)$ and $G(K)$, for some choice of representatives $\sigma_1,\dots,\sigma_N\in G(\Abf_K)$.  For each $1\leq n\leq N$, write $\sigma_n=(\sigma_{n,v})$, and let $S_0$ be a finite subset of places of $K$ containing $M_K^\infty$ such that $\sigma_{n,v}\in G(\widehat{\Ocal}_v)$ for all $1\leq n\leq N$ and all places $v\in M_K\setminus S_0$; such a finite set $S_0$ exists by the finiteness of the set $\{\sigma_1,\dots,\sigma_N\}$ and the definition of $G(\Abf_K)$ as a restricted direct product.  

Consider a finite subset $S$ of $M_K$ such that $S_0\subseteq S$.  For each $v\in M_K\setminus S$, let $\sigma_v$ be an element of $G(K)$, such that $\sigma_v\in G(\Ocal_v)$ for all except finitely many places $v$.  Arbitrarily selecting $\sigma_v\in G(\widehat{K}_v)$ for each $v\in S$ produces an adele $(\sigma_v)\in G(\Abf_K)$, and $(\ref{AdeleFinite})$ implies that $(\sigma_v)=(\delta_v) \cdot \sigma_n \cdot \sigma$ for some $(\delta_v)\in G^\infty(\Abf_K)$, some $1\leq n\leq N$, and some principal adele $\sigma\in G(K)$.  If $v\in M_K\setminus S$, then $v\notin S_0$, so  $\sigma\sigma_v^{-1}=\sigma_{n,v}^{-1}\delta_v^{-1} \in G(\widehat{\Ocal}_v)$, as desired.
\end{proof}

\begin{rem}
The result of Borel used in Lemma~\ref{ClassNumberLemma} holds more generally for arbitrary affine algebraic groups $G$, and can be viewed as an analogue for such groups of the finiteness of the class number of $K$.
\end{rem}

Let $S$ be a finite subset of $M_K$ containing $M_K^\infty$.  We say that two monic polynomials $F(x),G(x)\in\Ocal_S[x]$ of degree $N$ are {\em $\Ocal_S$-equivalent} if $F(x)= \alpha^{-N}G(\alpha x+\beta)$ for some $\alpha\in\Ocal_S^\times$ and $\beta\in\Ocal_S$.  The following is an affine variant of a finiteness result for binary forms due to Birch-Merriman \cite{MR0306119} and Evertse-Gy{\H{o}}ry \cite{MR1117339}.  To keep this paper as self contained as possible, we give a proof of the result using a fairly straightforward modification of the proof given in \cite{MR0306119}.  K. Gy{\H{o}}ry has pointed out to us that it can also be deduced in a few lines from Theorem 8 of \cite{MR727397}.

\begin{thm}\label{BirchMerrimanTheoremAffine}
Let $K$ be a number field, let $S$ be a finite subset of $M_K$ containing $M_K^\infty$, and let $N\geq2$ be an integer.  Then there exist only finitely many $\Ocal_S$-equivalence classes of monic polynomials $F(x)\in\Ocal_S[x]$ of degree $N$ with $\disc(F)\in\Ocal_S^\times$.
\end{thm}

\begin{proof}
Let $\Pi$ be the set of all monic polynomials $F(x)\in\Ocal_S[x]$ of degree $N$ with $\disc(F)\in\Ocal_S^\times$, and let $L$ be the splitting field of the set $\Pi$ over $K$.  Then $L/K$ is a finite extension; see for example \cite{bombierigubler} Cor. B.2.15.  Letting $T$ be the set of places of $L$ lying over the places of $K$ in $S$, we will first show that $\Pi$ is the union of finitely many $\Ocal_T$-equivalence classes.  Consider an arbitrary element $F(x)\in\Pi$, and let $e_1,\dots,e_N\in\Ocal_T$ denote the roots of $F(x)$; they are $T$-integral by Gauss's lemma.  Note also that $e_i-e_j\in\Ocal_T^\times$ whenever $i\neq j$, since $\disc(F)\in\Ocal_T^\times$.  The polynomial $F^*(x)=(e_2-e_1)^{-N}f((e_2-e_1)x+e_1)$ in $\Ocal_T[x]$ is monic and satisfies $f(0)=f(1)=0$, and thus
\begin{equation}\label{PolySpecialForm}
F^*(x) = x(x-1)(x-e^*_3)\dots(x-e^*_N)
\end{equation} 
for some $e^*_3\dots e^*_N\in\Ocal_T$.  In particular,
\begin{equation*}
\disc(F^*) = (e^*_3)^2\dots (e^*_N)^2(1-e^*_3)^2\dots (1-e^*_N)^2\prod_{3\leq i<j\leq N}(e^*_i-e^*_j)^2.
\end{equation*} 
Since $\disc(F^*)\in\Ocal_T^\times$, it follows that each pair $(e_j,1-e_j)$ is a solution in $(\Ocal_T^\times)^2$ to the unit equation $x+y=1$.  Since the are only finitely many such solutions (\cite{bombierigubler} $\S$5.1), there are only finitely many possibilities for $F^*(x)$, and since each $F(x)\in\Pi$ is $\Ocal_T$-equivalent to such a $F^*(x)$, we conclude that there are only finitely many $\Ocal_T$-equivalence classes of polynomials in $\Pi$.

To complete the proof, we have to show that each $\Ocal_T$-equivalence class in $\Pi$ is the union of finitely many $\Ocal_S$-equivalence classes.  Let $\Pi_0$ be an $\Ocal_T$-equivalence class in $\Pi$, and fix some $F_0(x)\in\Pi_0$; thus each $F(x)\in\Pi_0$ is equal to $\alpha^{-N}F_0(\alpha x+\beta)$ for some $\alpha\in\Ocal_T^\times$ and $\beta\in\Ocal_T$.  Denoting by $Z(F)$ and $Z(F_0)$ the set of roots of $F(x)$ and $F_0(x)$, respectively, we have a bijection $\sigma_{\alpha,\beta}:Z(F)\to Z(F_0)$ given by $\sigma_{\alpha,\beta}(x)=\alpha x+\beta$.  Enumerating $\Gal(L/K)=\{\tau_1,\dots,\tau_M\}$, each $\tau_m$ permutes the set $Z(F)$, and we obtain a bijection $i_{\alpha,\beta}:Z(F_0)^M\to Z(F_0)^M$ defined by
\begin{equation*}
i_{\alpha,\beta}(r_1,\dots,r_M)=(\sigma_{\alpha,\beta}\circ\tau_1\circ\sigma_{\alpha,\beta}^{-1}(r_1),\dots,\sigma_{\alpha,\beta}\circ\tau_M\circ\sigma_{\alpha,\beta}^{-1}(r_M)).
\end{equation*} 

Consider two polynomials in $\Pi_0$, say $F_1(x)=\alpha_1^{-N}F_0(\alpha_1 x+\beta_1)$ and $F_2(x)=\alpha_2^{-N}F_0(\alpha_2 x+\beta_2)$ for $\alpha_1,\alpha_2\in\Ocal_T^\times$ and $\beta_1,\beta_2\in\Ocal_T$.  We will show that, if $i_{\alpha_1,\beta_1}=i_{\alpha_2,\beta_2}$ as bijections $Z(F_0)^M\to Z(F_0)^M$, then $F_1(x)$ is $\Ocal_S$-equivalent to $F_2(x)$.  Since there are only finitely many bijections $Z(F_0)^M\to Z(F_0)^M$, it will follow that there are only finitely many $\Ocal_S$-equivalence classes in $\Pi_0$, completing the proof of the theorem.

Indeed, if $i_{\alpha_1,\beta_1}=i_{\alpha_2,\beta_2}$, then we let $\alpha=\alpha_1/\alpha_2$, and we let $\beta=(\beta_1-\beta_2)/\alpha_2$.  Then $\alpha\in\Ocal_T^\times$, $\beta\in\Ocal_T$, and $F_1(x)=\alpha^{-N}F_2(\alpha x+\beta)$.  Since $T$ is the set of places of $L$ lying over those places of $K$ in $S$, in order to show that $\alpha\in\Ocal_S^\times$ and $\beta\in\Ocal_S$ we just have to verify that $\alpha$ and $\beta$ are elements of $K$.  Fixing $\tau_m\in\Gal(L/K)$, the assumption that $i_{\alpha_1,\beta_1}=i_{\alpha_2,\beta_2}$ implies that
\begin{equation}\label{IdentityForEachm}
\sigma_{\alpha_1,\beta_1}\circ\tau_m\circ\sigma_{\alpha_1,\beta_1}^{-1}(r)=\sigma_{\alpha_2,\beta_2}\circ\tau_m\circ\sigma_{\alpha_2,\beta_2}^{-1}(r)
\end{equation} 
for each $r\in Z(F_0)$.  Since $\sigma_{\alpha_1,\beta_1}=\sigma_{\alpha_2,\beta_2}\circ\sigma_{\alpha,\beta}$, we deduce from $(\ref{IdentityForEachm})$ that 
\begin{equation*}
\sigma_{\alpha,\beta}\circ\tau_m(r)=\tau_m\circ \sigma_{\alpha,\beta}(r)
\end{equation*} 
for each $r\in Z(F_1)$.  This means that the two linear polynomials $\alpha x+\beta$ and $\tau_m(\alpha)x+\tau_m(\beta)$ take the same value for at least two distinct choices of $x$, namely the roots $r\in Z(F_1)$ of $F_1(x)$, from which we deduce that $\alpha x+\beta=\tau_m(\alpha)x+\tau_m(\beta)$, and therefore $\tau_m(\alpha)=\alpha$ and $\tau_m(\beta)=\beta$.  As $\tau_m\in\Gal(L/K)$ was arbitrary, we conclude that $\alpha,\beta\in K$, as desired. 
\end{proof}

\begin{prop}\label{GoodReductionProp}
Let $d\geq2$ be an integer, let $\lambda\in K^\times$, and let $\phi\in\Fcal_{d,\lambda}(K)$.
\begin{enumerate}
\item[(a)] Let $v\in M_K^0$ be a non-Archimedean place such that $\lambda\in\Ocal_v$.  Then $\phi$ has critically separable good reduction at $v$ if and only if $\phi$ is isomorphic to a rational map $\psi\in\Fcal_{d,\lambda}(K)$ given by $\psi(x)=A(x)/B(x)$, for a pair $(A(x),B(x))$ of polynomials in standard form with coefficients in $\Ocal_v$ and with $\Delta_{A,B}\in\Ocal_v^\times$.
\item[(b)] $\phi$ has critically separable good reduction at all except finitely many places $v\in M_K^0$.
\end{enumerate}
\end{prop}
\begin{proof}
(a) This follows at once from the definition of critically separable good reduction along with Proposition~\ref{GenDiscProp} (c).

(b)  Since $\phi$ is critically separable, it follows from Proposition~\ref{GenDiscProp} (c) that $\phi(x)=A(x)/B(x)$ for a pair $(A(x),B(x))$ of polynomials in standard form with coefficients in $K$ and with $\Delta_{A,B}\in K^\times$.  There exists a finite subset $S$ of $M_K$ containing $M_K^\infty$ such that $A(x)$ and $B(x)$ have coefficients in $\Ocal_S$, $\lambda\in \Ocal_S$, and $\Delta_{A,B}\in\Ocal_S^\times$.  By the definition of critically separable good reduction along with Proposition~\ref{GenDiscProp} (c), $\phi$ has critically separable good reduction at all $v\in M_K\setminus S$.
\end{proof}

According to Proposition~\ref{GoodReductionProp} (a), if a rational map $\phi\in\Fcal_{d,\lambda}(K)$ has critically separable good reduction at some place $v\in M_K^0$ such that $\lambda\in\Ocal_v$, then $\phi$ can be written as the ratio of two polynomials $A(x)$ and $B(x)$ possessing certain favorable local properties at the place $v$.  The following lemma, whose main technical ingredient is Lemma~\ref{ClassNumberLemma}, states that polynomials $A(x)$ and $B(x)$ can be found which enjoy these properties globally, at all places $v\in M_K\setminus S$, for sufficiently large subsets $S$ of $M_K$.

\begin{lem}\label{GlobalMinimalModel}
Given a number field $K$, an integer $d\geq2$, and an element $\lambda\in K^\times$, there exists a finite subset $S_0$ of $M_K$ containing $M_K^\infty$ with the following property.  If $S$ is a finite subset of $M_K$ containing $S_0$, and if $\phi\in\Fcal_{d,\lambda}(K)$ has critically separable good reduction at all places $v\in M_K\setminus S$, then there exists a rational map $\psi\in\Fcal_{d,\lambda}(K)$ which is isomorphic to $\phi$, such that $\psi(x)=A(x)/B(x)$ for a pair $(A(x),B(x))$ of polynomials in standard form with coefficients in $\Ocal_S$ and with $\Delta_{A,B}\in\Ocal_S^\times$.
\end{lem}

\begin{proof}
Taking $S_0$ large enough, we may assume that it contains the set $S_0$ whose existence is established in Lemma~\ref{ClassNumberLemma}, and that $\lambda\in\Ocal_{S_0}$ as well.  Let $S$ be a finite subset of $M_K$ such that $S_0\subseteq S$.  Thus $S$ satisfies the conclusion of Lemma~\ref{ClassNumberLemma}, and $\lambda\in\Ocal_{S}^\times$.

Consider a rational map $\phi\in\Fcal_{d,\lambda}(K)$ with critically separable good reduction at all places $v\in M_K\setminus S$.  We may write $\phi(x)=A_0(x)/B_0(x)$ for polynomials $A_0(x)$ and $B_0(x)$ in standard form, with coefficients in $K$ and with $\Delta_{A_0,B_0}\in K^\times$.    

For each place $v\in M_K\setminus S$, it follows from Proposition~\ref{GoodReductionProp} (a) that there exists a rational map $\psi_v\in\Fcal_{d,\lambda}(K)$ which is isomorphic to $\phi$, such that $\psi_v(x)=A_v(x)/B_v(x)$  for polynomials $A_v(x)$ and $B_v(x)$ in standard form, with coefficients in $\Ocal_v$ and with $\Delta_{A_0,B_0}\in \Ocal_v^\times$.   By the same argument given in the proof of Proposition~\ref{GoodReductionProp} (b), we may take $\psi_v=\phi$, $A_v(x)=A_0(x)$, and $B_v(x)=B_0(x)$ for all except finitely many places $v\in M_K\setminus S$.  

Since each $\psi_v$ is isomorphic to $\phi$, we have $\sigma_v\circ\phi\circ\sigma_v^{-1}=\psi_v$ for some $\sigma_v\in \Aut^\infty(\PP^1_K)$, with $\sigma_v(x)=x$ for all except finitely many places $v\in M_K\setminus S$.  It follows that $\psi_v(x) = A_0^{\sigma_v}(x)/B_0^{\sigma_v}(x)$, where the polynomials $A_0^{\sigma_v}(x)$ and $B_0^{\sigma_v}(x)$ are obtained from $A_0(x)$, $B_0(x)$, and $\sigma_v$ as in $(\ref{ABSigma})$.  Since $\psi_v(x) = A_v(x)/B_v(x)$ as well, and since both pairs $A_0^{\sigma_v}(x), B_0^{\sigma_v}(x)$ and $A_v(x), B_v(x)$ are in standard form, this implies that 
\begin{equation}\label{SigmavChange}
\begin{split}
A_v(x) & = A_0^{\sigma_v}(x), \\
B_v(x) & = B_0^{\sigma_v}(x),
\end{split}
\end{equation}
for all $v\in M_K\setminus S$.  

By Lemma~\ref{ClassNumberLemma} there exists some $\sigma\in \Aut^\infty(\PP^1_K)$ such that $\sigma\sigma_v^{-1}\in \Aut^\infty(\PP^1_{\Ocal_v})$ for all $v\in M_K\setminus S$.  Define 
\begin{equation}\label{SigmaChange}
\begin{split}
A(x) & = A_0^\sigma(x), \\
B(x) & = B_0^\sigma(x),
\end{split}
\end{equation}
where $A_0^{\sigma}(x)$ and $B_0^{\sigma}(x)$ are obtained from $A_0(x)$, $B_0(x)$, and $\sigma$ as in $(\ref{ABSigma})$.  Defining $\psi:\PP_K^1\to\PP_K^1$ by $\psi(x)=A(x)/B(x)$, plainly $\sigma\circ\phi\circ\sigma^{-1}=\psi$, so $\psi$ is isomorphic to $\phi$.

Given $v\in M_K\setminus S$, a calculation using $(\ref{SigmavChange})$ and $(\ref{SigmaChange})$ shows that 
\begin{equation}\label{SigmaSigmavChange}
\begin{split}
A(x) & = A_v^{\sigma\sigma_v^{-1}}(x), \\
B(x) & = B_v^{\sigma\sigma_v^{-1}}(x).
\end{split}
\end{equation}
Since both $A_v(x)$ and $B_v(x)$ have coefficients in $\Ocal_v$, and since $\sigma\sigma_v^{-1}\in\Aut^\infty(\PP_{\Ocal_v}^1)$, we conclude from $(\ref{SigmaSigmavChange})$ that both $A(x)$ and $B(x)$ have coefficients in $\Ocal_v$ as well.  Since $\Delta_{A_v,B_v}\in\Ocal_v^\times$, it follows from $(\ref{SigmaSigmavChange})$ and $(\ref{DeltaChangeSigma})$ that $\Delta_{A,B}\in\Ocal_v^\times$ as well.  Finally, since $v\in M_K\setminus S$ is arbitrary, we conclude that $A(x)$ and $B(x)$ have coefficients in $\Ocal_S$ and that $\Delta_{A,B}\in\Ocal_S^\times$.
\end{proof}

Given a rational map $\phi:\PP_K^1\to\PP_K^1$ of degree $d\geq2$, denote by $\Crit(\phi)$ the set of critical points of $\phi$ in $\PP^1(\Kbar)$.  

\begin{lem}\label{RigidityLemma}
Let $K$ be a number field, let $d\geq2$ be an integer, and let $\lambda\in K^\times$.  If $Z$ is a finite subset of $\PP^1(\Kbar)$, then there exist only finitely many rational maps $\phi\in\Fcal_{d,\lambda}(K)$ such that $\Crit(\phi)\subseteq Z$.
\end{lem}

\begin{proof}
Let $R_Z$ denote the set of all rational maps $\phi:\PP^1_K\to\PP^1_K$ of degree $d$ such that $\Crit(\phi)\subseteq Z$; we may assume that $R_Z$ is nonempty, since the lemma is trivial otherwise.  Since $\Crit(\sigma\circ\phi)=\Crit(\phi)$ for all rational maps $\phi:\PP^1_K\to\PP^1_K$ and all automorphisms $\sigma\in\Aut(\PP_K^1)$, we have a post-composition action $(\sigma,\phi)\mapsto\sigma\circ\phi$ of $\Aut(\PP_K^1)$ on $R_Z$.  Denote by $R_Z/\Aut(\PP_K^1)$ the set of orbits under this action, and given $\phi\in R_Z$, denote its orbit by $\langle\phi\rangle$.  Then:
\begin{enumerate}
\item[(i)] $R_Z$ is equal to a finite union of post-composition orbits $\langle\phi\rangle$;
\item[(ii)] $\langle\phi\rangle\cap \Fcal_{d,\lambda}(K)$ contains at most one element for each post-composition orbit $\langle\phi\rangle$.
\end{enumerate}  
Together, (i) and (ii) imply that $R_Z\cap\Fcal_{d,\lambda}(K)$ is finite, which is the desired result.

Assertion (i) is a classical fact going back to Schubert; see Goldberg \cite{MR1093002} for a sharp, quantitative version of this result.  To show (ii), suppose that both $\phi$ and $\sigma\circ\phi$ are elements of the family $\Fcal_{d,\lambda}(K)$; we must show that $\sigma$ is the identity element of $\Aut(\PP_K^1)$.  The fact that both $\phi$ and $\sigma\circ\phi$ fix $\infty$ implies that $\sigma$ fixes $\infty$; thus $\sigma(x)=\alpha x+\beta$ for some $\alpha\in K^\times, \beta\in K$.  Since $\infty$ is a fixed point of $\phi$ with multiplier $\lambda$, it is a fixed point of $\sigma\circ\phi$ with multiplier $\alpha^{-1}\lambda$.  But since $\sigma\circ\phi\in\Fcal_{d,\lambda}(K)$, we deduce that $\alpha^{-1}\lambda=\lambda$, whereby $\alpha=1$, and thus $\sigma(x)=x+\beta$.  Writing $\phi(x)$ as in $(\ref{RatMapAffine})$, the fact that both $\phi$ and $\sigma\circ\phi$ satisfy condition (F3) in the definition of the family $\Fcal_{d,\lambda}(K)$ means that both of the identities 
\begin{equation*}
\begin{split}
a_{d-1} & = \epsilon b_{d-2} \\
a_{d-1} +\beta \lambda& = \epsilon b_{d-2}
\end{split}
\end{equation*}
hold.  Since $\lambda\neq0$, subtracting the two identities we obtain $\beta=0$, and thus $\sigma(x)=x$, as desired.
\end{proof}

\begin{proof}[Proof of Theorem~\ref{FinitenessTheoremIntro}]  Enlarging the set $S$ only enlarges the set whose finiteness we are trying to prove, and so without loss of generality we may assume that $S$ contains the set $S_0$ of places whose existence is established in Lemma~\ref{GlobalMinimalModel}, and we may assume that $\lambda\in\Ocal_S^\times$.  

Suppose, contrary to the statement of the theorem, that there exists an infinite sequence $\{\phi_\ell\}$ ($\ell=1,2,3\dots$) of pairwise non-isomorphic rational maps in $\Fcal_{d,\lambda}(K)$ having critically separable good reduction at all places $v\in M_K\setminus S$.  Using Lemma~\ref{GlobalMinimalModel}, after possibly replacing each $\phi_\ell$ with another rational map in its isomorphism class, we may assume without loss of generality that $\phi_\ell(x)=A_\ell(x)/B_\ell(x)$, for polynomials $A_\ell(x)$ and $B_\ell(x)$ in standard form, with coefficients in $\Ocal_S$ and with $\Delta_{A_\ell,B_\ell}\in\Ocal_S^\times$.

For each $\ell$, define $f_\ell(x) = \lambda^{-1}W_{A_\ell,B_\ell}(x)$.  Then $f_\ell(x)\in\Ocal_S[x]$ is monic, vanishes precisely at the critical points of $\phi_\ell$ in $\Kbar$, and satisfies $\disc(f_\ell)\in\Ocal_S^\times$.  According to Theorem~\ref{BirchMerrimanTheoremAffine}, after passing to an infinite subsequence of $\{\phi_\ell\}$, we may assume without loss of generality that each $f_\ell(x)$ is $\Ocal_S$-equivalent to $f_1(x)$.  This means that for each $\ell$, $f_\ell(x)=\alpha_\ell^{-(2d-2)} f_1(\alpha_\ell x+\beta_\ell)$ for some $\alpha_\ell\in\Ocal_S^\times$ and $\beta_\ell\in\Ocal_S$.  Defining $\sigma_\ell\in\Aut^\infty(\PP^1_K)$ by $\sigma_\ell(x)=\alpha_\ell x+\beta_\ell$, and letting $\psi_\ell=\phi_\ell^{\sigma_\ell}=\sigma_\ell\circ\phi_\ell\circ\sigma_\ell^{-1}$, it follows that $\Crit(\psi_\ell)=\Crit(\phi_1)$ for all $\ell$.  

We have produced an infinite sequence $\{\psi_\ell\}$ of distinct rational maps in $\Fcal_{d,\lambda}(K)$ having the same set of critical points.  This violates Lemma~\ref{RigidityLemma}, and the contradiction completes the proof.
\end{proof}

\begin{rem}  Our interest in the family $\Fcal_{d,\lambda}(K)$ is motivated by an attempt to give a natural generalization of the family $\Lcal(K)$ of Latt\`es maps.  However, it is not hard to modify conditions (F2) and (F3) to produce other potentially interesting families of critically separable rational maps for which the methods of this paper apply.

For example, fix an integer $d\geq2$ and and element $\lambda\in K^\times$, and define $\Fcal(K)$ to be the set of all critically separable rational maps of degree $d$ defined over $K$ such that $\infty$ is a fixed point of $\phi$ with multiplier $\lambda$, and such that $0$ is a fixed point of $\phi$ (with arbitrary multiplier).  Observe that the family $\Fcal(K)$ is closed under conjugation by the group
\begin{equation*}
\Gcal(K)=\{\sigma\in\Aut(\PP_K^1)\mid\sigma(x)=\alpha x\text{ for some } \alpha\in K^\times\}.
\end{equation*}
Define {\em $K$-isomorphism} between two rational maps in the family $\Fcal(K)$ via $\Gcal(K)$-conjugation, and declare that a rational map $\phi\in\Fcal(K)$ has {\em critically separable good reduction} at a non-Archimedean place $v$ of $K$ if $\phi$ is $K$-isomorphic to a $v$-integral rational map $\psi\in\Fcal(K)$ such that the reduced rational map $\tilde{\psi}_v:\PP^1_{ k_v}\to\PP^1_{ k_v}$ has degree $d$ and is critically separable.  It is not hard to see that the intersection $\langle\phi\rangle\cap\Fcal(K)$ contains at most one rational map for each $\phi\in\Fcal(K)$, where $\langle\phi\rangle$ denotes the orbit of $\phi$ under the post-composition action of $\Aut(\PP^1_K)$ (in fact, it is enough to know that this intersection is finite); this observation is required for the family $\Fcal(K)$ to satisfy the statement of Lemma~\ref{RigidityLemma}.  

It follows from a straightforward modification of the proof of Theorem~\ref{FinitenessTheoremIntro} that for each finite subset $S$ of $M_K$ containing $M_K^\infty$, the family $\Fcal(K)$ contains only finitely many $K$-isomorphism classes of rational maps having critically separable good reduction at all places $v\not\in S$.
\end{rem}

\section{The Minimal Critical Discriminant}\label{MinCritDisc}

Given an elliptic curve $E/K$, its {\em minimal discriminant} $\Delta(E)$ is a certain integral ideal of $\Ocal_K$ which can be viewed as a global measure of the arithmetic complexity of the curve.  Explicitly,
\begin{equation*}
\Delta(E) = \prod_{v\in M_K^0}\pfrak_v^{\delta_v(E)},
\end{equation*}
where for each non-Archimedean place $v\in M_K^0$, $\pfrak_v$ denotes the associated prime ideal of $\Ocal_K$, and the exponent $\delta_v(E)$ is defined to be the minimal $v$-adic valuation $\ord_v(\Delta)$ over the discriminants $\Delta$ of all $v$-integral Weierstrass equations for $E$ over $K$.  

It follows from Shafarevich's theorem that the norm $\NN_{K/\QQ}(\Delta(E))$ of the minimal discriminant is bounded above by a quantity depending on the number field $K$ and on the set of places at which $E/K$ has bad reduction, but not depending otherwise on the curve $E$.  The following well-known conjecture of Szpiro would give one possible quantitative version of this bound.  Given an ideal $\afrak$ of $\Ocal_K$, define its {\em radical} to be the squarefree product $\Rfrak(\afrak)=\prod_{\pfrak\mid\afrak}\pfrak$ of the prime ideals dividing it.  In particular, $\Rfrak(\Delta(E))$ is simply the squarefree product of the prime ideals $\pfrak_v$ at which $E/K$ has bad reduction.
 
\begin{SzpConj}[\cite{MR1065151}]\label{SzpiroConjecture}
Let $K$ be a number field and let $\epsilon>0$.  Then
\begin{equation}\label{SzpiroIneq}
\NN_{K/\QQ}(\Delta(E)) \ll_{K,\epsilon} \NN_{K/\QQ}(\Rfrak(\Delta(E)))^{6+\epsilon}
\end{equation}
for all semistable elliptic curves $E/K$.
\end{SzpConj}

Recall that $E/K$ is said to be {\em semistable} if it has either good or multiplicative reduction at all places $v\in M_K^0$.  (Szpiro's Conjecture can be stated without the semistable requirement, provided that the squarefree radical $\Rfrak(\Delta(E))$ is replaced with the conductor of $E/K$, a more complicated invariant which we do not need to consider in this paper.)  Szpiro's conjecture for $K=\QQ$ is closely related to the $abc$ conjecture of Masser-Oesterl\'e (see \cite{bombierigubler} $\S$12.5), and a proof of Szpiro's conjecture would also have a number of interesting consequences concerning the arithmetic of elliptic curves; see for example \cite{hindrysilverman:integralpts}, \cite{MR2259240}.

In this section we formulate a conjecture which bears roughly the same relationship to Theorem~\ref{FinitenessTheoremIntro} as Szpiro's conjecture bears to Shafarevich's theorem.  Again let $K$ be a number field, let $d\geq2$ be an integer, let $\lambda\in K^\times$, and denote by $S_\lambda$ the (finite) set of places of $K$ which are either Archimedean or for which $\lambda\not\in\Ocal_v$.  

Given a rational map $\phi\in\Fcal_{d,\lambda}(K)$ and a place $v\in M_K\setminus S_\lambda$, define $\delta_v(\phi)$ to be the minimal value of $\ord_v(\Delta_{A,B})$ over all pairs $(A(x),B(x))$ of polynomials in standard form with coefficients in $\Ocal_v$, such that the rational map $\psi:\PP_K^1\to\PP_K^1$ given by $\psi(x)=A(x)/B(x)$ is isomorphic to $\phi$.  Since the critical discriminant $\Delta_{A,B}$ is an integral polynomial in the coefficients of $A(x)$ and $B(x)$, it follows that $\ord_v(\Delta_{A,B})\geq0$ for all such pairs $(A(x),B(x))$, and therefore $\delta_v(\phi)$ is a nonnegative integer.  Define the {\em minimal critical discriminant} of $\phi$ to be the integral ideal of $\Ocal_K$ given by
\begin{equation*}
\Delta(\phi) = \prod_{v\in M_K\setminus S_\lambda}\pfrak_v^{\delta_v(\phi)}.
\end{equation*}
Thus $\Delta(\phi)$ is supported precisely on the set of places $v\in M_K\setminus S_\lambda$ at which $\phi$ has critically separable bad reduction.    

\begin{conj}\label{CondDiscConjecture}
Let $K$ be a number field, let $d\geq3$ be an integer, let $\lambda\in K^\times$, and let $\epsilon>0$.  Then
\begin{equation}\label{ConjIneq}
\NN_{K/\QQ}(\Delta(\phi)) \ll_{K,d,\lambda,\epsilon} \NN_{K/\QQ}(\Rfrak(\Delta(\phi)))^{(2d-2)(2d-3)+\epsilon}
\end{equation}
for all $\phi\in\Fcal_{d,\lambda}(K)$.
\end{conj}

The conjectural exponent of $(2d-2)(2d-3)+\epsilon$ is suggested by the analogy with Szpiro's conjecture, along with the identity $(\ref{DeltaChangeSigma})$.  Given a place $v\in M_K\setminus S_\lambda$ and a rational map $\psi(x)=A(x)/B(x)$ which is isomorphic to $\phi$, where $(A(x),B(x))$ is a pair of polynomials in standard form, with coefficients in $\Ocal_v$, we have $\ord_v(\Delta_{A,B})\geq0$, and the identity $(\ref{DeltaChangeSigma})$ implies that $\ord_v(\Delta_{A,B})$ is well-defined (independent of $\psi$) modulo $(2d-2)(2d-3)$.  It follows that
\begin{equation}\label{MinValueHeuristic}
\ord_v(\Delta_{A,B})<(2d-2)(2d-3) \hskip5mm \Longrightarrow \hskip5mm \delta_v(\phi)=\ord_v(\Delta_{A,B}).
\end{equation}
The converse of $(\ref{MinValueHeuristic})$ need not hold, but Conjecture~\ref{CondDiscConjecture} predicts that it almost holds in the average over all places $v\in M_K\setminus S_\lambda$; that is, the conjecture implies that $\delta_v(\phi)$ is globally not often larger than $(2d-2)(2d-3)$ as $v$ ranges over all places in $M_K\setminus S_\lambda$.

In view of the correspondence between elliptic curves and and Latt\`es maps, it should not come as a surprise to find a close relationship between Szpiro's conjecture and Conjecture~\ref{CondDiscConjecture}.  In Theorem~\ref{ImpliesSzpiro} we will use the fact that the family $\Lcal(K)$ of Latt\`es maps is contained in the family $\Fcal_{4,4}(K)$ to show that Conjecture~\ref{CondDiscConjecture} (in the special case $d=\lambda=4$) implies Szpiro's conjecture for semistable elliptic curves.  We will first need two technical results.

\begin{prop}\label{LattesInvProp}
The family $\Lcal(K)$ of Latt\`es maps defined in $\S$\ref{Background} is invariant under $\Aut^\infty(\PP_K^1)$-conjugation.  More precisely, let $f(x)=x^3+ax^2+bx+c$ be a monic polynomial in $K[x]$ with distinct roots, let $\phi_{a,b,c}\in\Lcal(K)$ be the associated Latt\`es map defined in $\S$\ref{Background}, and let $\sigma\in\Aut^\infty(\PP_K^1)$ be an automorphism given by $\sigma(x)=\alpha x+\beta$ for $\alpha\in K^\times$ and $\beta\in K$.  Then $\sigma\circ\phi_{a,b,c}\circ\sigma^{-1}=\phi_{a^*,b^*,c^*}$, where the polynomial $f^*(x)=x^3+a^*x^2+b^*x+c^*$ is defined by $f^*(x)=\alpha^3f(\alpha^{-1}(x-\beta))$.
\end{prop}
\begin{proof}
We omit this calculation, which is elementary.
\end{proof}

\begin{prop}\label{LattesIsomProp}
Let $E/K$ and $E^*/K$ be elliptic curves given by Weierstrass equations $y^2=x^3+ax^2+bx+c$ and $y^2=x^3+a^*x^2+b^*x+c^*$ over $K$, respectively, and let $\phi_{a,b,c}, \phi_{a^*, b^*, c^*}\in\Fcal_{4,4}(K)$ be the corresponding Latt\`es maps defined in $\S$\ref{Background}. 
\begin{itemize}
	\item[(a)] If $E$ is isomorphic to $E^*$ over $K$, then $\phi_{a,b,c}$ is isomorphic to $\phi_{a^*, b^*, c^*}$ over $K$.
	\item[(b)] If $\phi_{a,b,c}$ is isomorphic to $\phi_{a^*, b^*, c^*}$ over $K$, then there exists an extension $K'/K$ of degree at most $2$ such that $E$ is isomorphic to $E^*$ over $K'$.
\end{itemize}
\end{prop}
\begin{proof}
(a) An isomorphism $E\to E^*$ over $K$ must take the form $(x,y)\mapsto(\alpha^2x+\beta,\alpha^3y)$ for $\alpha\in K^\times, \beta\in K$; see \cite{silverman:aec} $\S$III.1.  Writing $X=\alpha^2x+\beta$ and $Y=\alpha^3y$, and letting $f^*(X)=X^3+a^*X^2+b^*X+c^*$, it follows that $f^*(X)=\alpha^6f(\alpha^{-2}(X-\beta))$.  Proposition~\ref{LattesInvProp} then implies that $\sigma\circ\phi_{a,b,c}\circ\sigma^{-1}=\phi_{a^*,b^*,c^*}$, where $\sigma(x)=\alpha^2x+\beta$.

(b)  If $\phi_{a,b,c}$ is isomorphic to $\phi_{a^*, b^*, c^*}$ over $K$, then $\sigma\circ\phi_{a,b,c}\circ\sigma^{-1}=\phi_{a^*,b^*,c^*}$ for some $\sigma\in\Aut^\infty(\PP_K^1)$ given by $\sigma(x)=\alpha x+\beta$, where $\alpha\in K^\times$ and $\beta\in K$.  Let $\alpha_0=\sqrt{\alpha}$ and let $K'=K(\alpha_0)$.  The map $(x,y)\mapsto(\alpha_0^2x+\beta,\alpha_0^3y)$ defines an isomorphism $E\to E^*$ over $K'$.
\end{proof}

\begin{thm}\label{ImpliesSzpiro}
Conjecture~\ref{CondDiscConjecture} for the family $\Fcal_{4,4}(K)$ implies Szpiro's conjecture for semistable elliptic curves. 
\end{thm}
\begin{proof}
Let $E/K$ be a semistable elliptic curve given by a Weierstrass equation $y^2=x^3+ax^2+bx+c$ with discriminant $\Delta_E$, and let $\phi_{a,b,c}\in\Fcal_{4,4}(K)$ be the corresponding Latt\`es map defined in $\S$\ref{Background}.  Then $\phi(x)=A(x)/B(x)$ for polynomials $A(x)=x^4-2bx^2-8cx+b^2-4ac$ and $B(x)=4x^3+4ax^2+4bx+4c$, and 
\begin{equation}\label{TwoDiscriminants}
\Delta_{A,B} = -2^{18}\Delta_E^5,
\end{equation}
which follows from $(\ref{EllipticDisc})$ and $(\ref{CritDiscLattes})$.

We will show that
\begin{equation}\label{TwoIneq}
\begin{split}
\NN_{K/\QQ}(\Delta(E))^{5} & \ll \NN_{K/\QQ}(\Delta(\phi_{a,b,c})) \\ 
\NN_{K/\QQ}(\Rfrak(\Delta(\phi_{a,b,c}))) & \ll \NN_{K/\QQ}(\Rfrak(\Delta(E)))
\end{split}
\end{equation}
(with implied constants depending only on $K$).   When $d=4$, we have $(2d-2)(2d-3)=30$, and so together the two inequalities $(\ref{TwoIneq})$ show that $(\ref{ConjIneq})$ implies $(\ref{SzpiroIneq})$.

To prove the second inequality in $(\ref{TwoIneq})$, consider a place $v\in M_K^0$ of residue characteristic not equal to $2$ or $3$.  If $E/K$ has good reduction at $v$, then $E$ is isomorphic over $K$ to an elliptic curve $E^*/K$ given by a $v$-integral Weierstrass equation $y^2=x^3+a^*x^2+b^*x+c^*$ with discriminant $\Delta_{E^*}\in\Ocal_v^\times$.  According to Proposition~\ref{LattesIsomProp} (a), $\phi_{a,b,c}$ is isomorphic to $\phi_{a^*,b^*,c^*}$, and using $(\ref{TwoDiscriminants})$ and Proposition~\ref{GoodReductionProp} (a) we conclude that $\phi_{a,b,c}$ has critically separable good reduction at $v$.  We have shown that the squarefree integral ideal $\Rfrak(\Delta(\phi_{a,b,c}))$ is divisible only by primes $\pfrak_v$ lying over $2$ or $3$ or for which $\pfrak_v\mid \Rfrak(\Delta(E))$.  It follows that $\NN_{K/\QQ}(\Rfrak(\Delta(\phi_{a,b,c}))) \ll \NN_{K/\QQ}(\Rfrak(\Delta(E)))$.

To prove the first inequality in $(\ref{TwoIneq})$, we will show that 
\begin{equation}\label{ImpliesSzpiroLocalIneq}
5\delta_v(\Delta_{E})\leq \ord_v(2^{-18})+\delta_v(\phi_{a,b,c})
\end{equation}
for all places $v\in M_K^0$.  Assembling the local inequalities $(\ref{ImpliesSzpiroLocalIneq})$ into a global inequality we obtain the first inequality in $(\ref{TwoIneq})$.

It remains only to prove $(\ref{ImpliesSzpiroLocalIneq})$.  Fix a place $v\in M_K^0$.  If $E/K$ has good reduction at $v$ then $\delta_v(\Delta_{E})=0$, and so $(\ref{ImpliesSzpiroLocalIneq})$ holds trivially.  By the semistable assumption it now suffices to consider the case that $E/K$ has multiplicative reduction at $v$.  This means that $E$ is isomorphic over $K$ to an elliptic curve $E_\min/K$ given by a $v$-integral Weierstrass equation
\begin{equation}\label{MinimalWeierstrass}
y^2 +a_1xy +a_3y = x^3 + a_2x^2 + a_4x + a_6
\end{equation}
for which $c_4$ is a $v$-adic unit.  Here $c_4$ is a standard expression in the coefficients $a_j$, and it is related to the $j$-invariant associated to this isomorphism class of elliptic curves by $j=c_4^3/\Delta_{E_\min}$; see \cite{silverman:aec} $\S$III.1 for the precise definition.  Since the $j$-invariant is an isomorphism invariant and $\ord_v(c_4)=0$, it follows that $(\ref{MinimalWeierstrass})$ is in fact a minimal Weierstrass equation for $E$.  Thus $\delta_v(E)=\ord_v(\Delta_{E_\min})$.  

Now let $\phi_{a^*, b^*, c^*}\in\Fcal_{4,4}(K)$ be a Latt\`es map which is isomorphic to $\phi_{a, b, c}$ and given by $\phi_{a^*, b^*, c^*}(x)=A^*(x)/B^*(x)$ for a pair $(A^*(x),B^*(x))$ of polynomials in standard form, such that $a^*,b^*,c^*\in\Ocal_v$, and such that $\ord_v(\Delta_{A^*,B^*})$ is minimal among all such rational maps in $\Fcal_{4,4}(K)$.  Thus $\delta_v(\phi_{a,b,c})=\ord_v(\Delta_{A^*,B^*})$.  

Denote by $E^*/K$ the elliptic curve given by the Weierstrass equation 
\begin{equation}\label{LattesMinimalWeierstrass}
y^2=x^3+a^*x^2+b^*x+c^*.
\end{equation}
It follows from Proposition~\ref{LattesIsomProp} that the elliptic curves $E_\min$ and $E^*$ are isomorphic over $\Kbar$.  In particular, both curves have the same $j$-invariant, which implies that $c_4^3/\Delta_{E_\min}=(c_4^*)^3/\Delta_{E^*}$, where $c_4^*$ denotes the usual expression associated to the Weierstrass equation $(\ref{LattesMinimalWeierstrass})$. Rearranging we have $\Delta_{E^*}=(c_4^*)^3c_4^{-3}\Delta_{E_\min}$, and therefore $\ord_v(\Delta_{E^*})\geq\ord_v(\Delta_{E_\min})$, since $c_4$ is a $v$-adic unit and $c_4^*$ is $v$-integral.  Finally, using the identity $(\ref{TwoDiscriminants})$ we have
\begin{equation*}
5\ord_v(\Delta_{E_\min})\leq 5\ord_v(\Delta_{E^*})=\ord_v(2^{-18})+\ord_v(\Delta_{A^*,B^*}),
\end{equation*}
which implies $(\ref{ImpliesSzpiroLocalIneq})$, because $\delta_v(E)=\ord_v(\Delta_{E_\min})$ and $\delta_v(\phi_{a,b,c})=\ord_v(\Delta_{A^*,B^*})$.
\end{proof}

\begin{rem}  As the reader may have observed, Conjecture~\ref{CondDiscConjecture} is stated only for $d\geq3$.  In fact, the statement of the conjecture holds when $d=2$, but for a somewhat trivial reason following from a purely local argument.  Each isomorphism class in $\Fcal_{2,\lambda}(K)$ contains a rational map of the form $\phi(x)=A(x)/B(x)$ for polynomials $A(x)=x^2+a$ and $B(x)=\lambda x$, where $a\neq0$.  Given a place $v\in M_K\setminus S_\lambda$, let $\pi_v\in K$ be a uniformizer at $v$, and let $m$ be the (unique) integer such that $0\leq\ord_v(\pi_v^{2m}a)\leq1$.  Letting $\sigma(x)=\pi_v^mx$, we have $\sigma\circ\phi\circ\sigma^{-1}(x)=A^\sigma(x)/B^\sigma(x)$ for $v$-integral polynomials $A^\sigma(x)=x^2+\pi_v^{2m}a$ and $B^\sigma(x)=\lambda x$, and the critical discriminant is given by $\Delta_{A^\sigma,B^\sigma}=4\lambda^2\pi_v^{2m}a$.  We conclude that $\delta_v(\phi) \leq \ord_v(\Delta_{A^\sigma,B^\sigma})\leq \ord_v(4\lambda^2)+1$.  Since $\delta_v(\phi)=0$ as all places of critically separable good reduction, we conclude that $\NN_{K/\QQ}(\Delta(\phi)) \ll \NN_{K/\QQ}(\Rfrak(\Delta(\phi)))$.  
\end{rem}



\medskip

\begin{thebibliography}{10}

\bibitem{MR0306119}
{\sc Birch, B.~J., and Merriman, J.~R.}
\newblock Finiteness theorems for binary forms with given discriminant.
\newblock {\em Proc. London Math. Soc. (3) 24\/} (1972), 385--394.

\bibitem{bombierigubler}
{\sc Bombieri, E., and Gubler, W.}
\newblock {\em Heights in {D}iophantine Geometry}.
\newblock No.~4 in New Mathematical Monographs. Cambridge University Press,
  Cambridge, 2006.

\bibitem{MR0202718}
{\sc Borel, A.}
\newblock Some finiteness properties of adele groups over number fields.
\newblock {\em Inst. Hautes \'Etudes Sci. Publ. Math.}, 16 (1963), 5--30.

\bibitem{MR1117339}
{\sc Evertse, J.-H., and Gy{\H{o}}ry, K.}
\newblock Effective finiteness results for binary forms with given
  discriminant.
\newblock {\em Compositio Math. 79}, 2 (1991), 169--204.

\bibitem{MR718935}
{\sc Faltings, G.}
\newblock Endlichkeitss\"atze f\"ur abelsche {V}ariet\"aten \"uber
  {Z}ahlk\"orpern.
\newblock {\em Invent. Math. 73}, 3 (1983), 349--366.

\bibitem{MR1093002}
{\sc Goldberg, L.~R.}
\newblock Catalan numbers and branched coverings by the {R}iemann sphere.
\newblock {\em Adv. Math. 85}, 2 (1991), 129--144.

\bibitem{MR727397}
{\sc Gy{\H{o}}ry, K.}
\newblock Effective finiteness theorems for polynomials with given discriminant
  and integral elements with given discriminant over finitely domains.
\newblock {\em J. Reine Angew. Math. 346\/} (1984), 54--100.

\bibitem{hindrysilverman:integralpts}
{\sc Hindry, M., and Silverman, J.~H.}
\newblock The canonical height and integral points on elliptic curves.
\newblock {\em Invent. Math. 93}, 2 (1988), 419--450.

\bibitem{MR828821}
{\sc Mazur, B.}
\newblock Arithmetic on curves.
\newblock {\em Bull. Amer. Math. Soc. (N.S.) 14}, 2 (1986), 207--259.

\bibitem{MR2259240}
{\sc Petsche, C.}
\newblock Small rational points on elliptic curves over number fields.
\newblock {\em New York J. Math. 12\/} (2006), 257--268 (electronic).

\bibitem{silverman:aec}
{\sc Silverman, J.~H.}
\newblock {\em The {A}rithmetic of {E}lliptic {C}urves}, vol.~106 of {\em
  Graduate Texts in Mathematics}.
\newblock Springer-Verlag, New York, 1992.
\newblock Corrected reprint of the 1986 original.

\bibitem{MR2316407}
{\sc Silverman, J.~H.}
\newblock {\em The {A}rithmetic of {D}ynamical {S}ystems}, vol.~241 of {\em
  Graduate Texts in Mathematics}.
\newblock Springer, New York, 2007.

\bibitem{MR1065151}
{\sc Szpiro, L.}
\newblock Discriminant et conducteur des courbes elliptiques.
\newblock {\em Ast\'erisque}, 183 (1990), 7--18.
\newblock S{\'e}minaire sur les Pinceaux de Courbes Elliptiques (Paris, 1988).

\bibitem{arxiv1010.5030}
{\sc Szpiro, L., Tepper, M., and Williams, P.}
\newblock Resultant and conductor of geometrically semi-stable self maps of the
  projective line over a number field or function field, 2011.
\newblock arXiv:1010.5030.

\bibitem{MR2435841}
{\sc Szpiro, L., and Tucker, T.~J.}
\newblock A {S}hafarevich-{F}altings theorem for rational functions.
\newblock {\em Pure Appl. Math. Q. 4}, 3, part 2 (2008), 715--728.

\end{thebibliography}
\end{document}